%% file: main.tex
\documentclass[runningheads]{llncs}
\usepackage[T1]{fontenc}

\usepackage{soul}
\usepackage{amsmath}
\usepackage{longtable}
\usepackage{booktabs,multirow,multicol,adjustbox,amsmath,amsfonts,mathtools,algorithm,algorithmic,comment}
\usepackage{pdflscape}

\usepackage{bbding}

\usepackage{amssymb}

\usepackage{graphics}
\usepackage{bbm}
\usepackage{tikz}
\usepackage{bm}
\usepackage{appendix}
\usepackage{hyperref}

\usepackage{graphicx}
\usepackage{subfig}

\usepackage{multirow}
\usepackage{rotating}
\usepackage{verbatim}
\usepackage{optidef}
\usepackage{cleveref}
\usepackage{array}
\usepackage{tabularx}
\usepackage{arydshln}
\usepackage{easybmat}
\usepackage{enumitem}
\usepackage{xcolor}
\usetikzlibrary{positioning}
\newcommand*\xbar[1]{%
  \hbox{%
    \vbox{%
      \hrule height 0.5pt 
      \kern0.5ex
      \hbox{%
        \kern-0.1em
        \ensuremath{#1}%
        \kern-0.1em
      }%
    }%
  }%
} 

\usepackage[T1]{fontenc}
\usepackage{graphicx}
\begin{document}

\title{Approximating the Gomory Mixed-Integer Cut Closure Using Historical Data}

\titlerunning{Approximating GMIC Closure Using Historical Data}

\author{}

\author{
   Berkay Becu\inst{1} \and
   Santanu S. Dey\inst{1} \and
   Feng Qiu\inst{2} \and
   \'Alinson S. Xavier\inst{2}
}

\authorrunning{Becu, Dey, Qiu \& Xavier}
\institute{Georgia Institute of Technology, Atlanta, GA, USA \email{bbecu3@gatech.edu,santanu.dey@isye.gatech.edu} \and Argonne National Laboratory, Lemont, IL, USA \email{\{fqiu,axavier\}@anl.gov}}

\maketitle


\begin{abstract}
    Many operations related optimization problems involve repeatedly solving similar mixed integer linear programming (MILP) instances with the same constraint matrix but differing objective coefficients and right-hand-side values.
    The goal of this paper is to generate good cutting-planes for such instances using historical data.
    Gomory mixed integer cuts (GMIC) for a general MILP can be parameterized by a vector of weights to aggregate the constraints into a single equality constraint, where each such equality constraint in turn yields a unique GMIC.
    In this paper, we prove that for a family of MILP instances, where the right-hand-side of the instances belongs to a lattice, the GMIC closure for every instance in this infinite family can be obtained using the same finite list of aggregation weights. 
    This result motivates us to build a simple heuristic to efficiently select aggregations for generating GMICs from historical data of similar instances with varying right-hand-sides and objective function coefficients. 
    For testing our method, we generated families of instances by perturbing the right-hand-side and objective functions of MIPLIB 2017 instances. 
    The proposed heuristic
can significantly accelerate the performance of Gurobi for many benchmark instances, even when taking into account the time required to predict aggregation
multipliers and compute the cut coefficients. 
To the best of our knowledge, this
is the first work in the literature of data-driven cutting plane generation that is
able to significantly accelerate the performance of a commercial state-of-the-art
MILP solver, using default solver settings, on large-scale benchmark instances.

\end{abstract}


\section{Introduction}

Mixed-Integer Linear Programming (MILP) models are essential for solving optimization problems across various industrial sectors, including logistics, manufacturing and power systems.
In many practical settings, decision-makers must solve challenging instances under strict time limits, as the data needed to formulate the problem often becomes available only shortly before a solution is required.
For example, the unit commitment problem, which Independent System Operators solve daily to clear day-ahead electricity markets, requires solutions within minutes, since operators only have a few hours to clear the markets after receiving bids from participants~\cite{xavier2021learning}.
Further applications are discussed in~\cite{johnson2020k,xiaoyi2022exploiting,kocuk2016cycle}.

Operational MILPs often exhibit two key characteristics that suggest a solution approach leveraging historical data.
First, these problems are usually solved on a recurring basis --- often hourly or daily --- creating an opportunity for data collection, analysis and offline processing, which traditional MILP research, focused on one-shot solution methods, has not taken advantage of.
Second, instances of these problems tend to be structurally similar: decisions are made over the same set of variables, and the constraint matrix typically remains the same, as it encodes decision logic or invariant physical resources.
What usually varies between instances are constraint right-hand-sides and objective function coefficients, reflecting fluctuations for example in costs, revenues, demands and availability of different resources.
Thus, while each instance may be challenging to solve independently, from scratch, if instances are viewed as random samples drawn from a particular probability distribution, it may be feasible to leverage offline processing and historical information to accelerate the solution of new instances drawn from this distribution.

Many papers have proposed data-driven methods to construct high-quality feasible solutions~\cite{xavier2021learning,bertsimas2022online,johnson2020k} and improve the performance of branching rules in branch-and-bound methods \cite{alvarez2017machine,nair2020solving,gupta2020hybrid,xiaoyi2022exploiting}, as well as conduct sensitivity analysis with respect to changing right-hand-sides~\cite{cifuentes2024sensitivity}.
Recent work has also explored data-driven cutting plane generation, which we consider in this manuscript.
In the following, we provide only a brief review of this literature, and we refer to \cite{deza2023machine} for a more complete survey.
One of the pioneering works in this area is by Tang et al.~\cite{tang2020reinforcement}, which models the cut selection process as a Markov decision process and proposes the use of reinforcement learning.
Paulus et al.~\cite{paulus2022learning} apply imitation learning instead, where a greedy algorithm acts as the expert to select an individual cut that maximizes gap closure.
Several studies have trained neural networks to select multiple cuts simultaneously~\cite{huang2022learning,wang2023learning}, instead of adding one cut at a time.
Other research~\cite{turner2022adaptive,berthold2022learning} focus on learning cut-related solver parameters tailored to specific instances.
On the theoretical side, Balcan et al.~\cite{balcan2022structural} derive upper bounds on the sample size required to estimate the expected branch-and-bound tree size when one \emph{Gomory mixed integer cut} (GMIC) is added, Cheng et al.~\cite{cheng2024blearning} further study sample complexity for learning multiple cutting planes using neural networks, and Cheng et al.~\cite{cheng2024learning} provide such bounds for  learning general group cuts.
Several methods consider generating cutting-planes directly using machine learning.
Chetelat et al.~\cite{chetelat2023continuous} show how the separation problem of GMI cuts can be reformulated as a training problem using ReLU activation functions.
Dragotto et al.~\cite{dragotto2023differentiable} introduce a ``cutting-plane layer'', a differentiable generator that maps problem data and previous iterates to cutting planes.
Finally, and most recently, Guage et al.~\cite{guaje2024machine} present a hybrid framework where an ML model simplifies the cut-generating optimization problem.

In this paper, our main goal is to show how one may exploit historical data to generate effective cutting planes for families of similar MILP instances.
More specifically, we aim to accelerate MILP solver performance by learning to generate cuts that approximate the Gomory Mixed-Integer Cut (GMIC) closure for families of instances with same constraint matrix, but varying right-hand-sides and objective functions.
As we further discuss in Section~\ref{sec:cutclosure}, GMIC cuts were among the first cutting planes developed for solving MILPs and their closure has many desirable theoretical and computational properties.
Concretely, the GMIC closure is known to be polyhedral~\cite{cook1990chvatal}, rank-1 GMI cuts are substantially more numerically stable than higher-rank cuts, and computational studies on MIPLIB instances have shown that they effectively close a significant portion of the integrality gap~\cite{balas2008optimizing,dash2010heuristic}.
However, exact separation over the GMIC closure is NP-hard~\cite{caprara2003separation,lee2019np}, and heuristics such as those in~\cite{dash2010heuristic,fischetti2013approximating} remain too computationally expensive to be integrated into state-of-the-art commercial MILP solvers.

We have two main contributions in this manuscript.
First, in Section~\ref{sec:main}, we show a new finiteness result regarding GMI closures for infinite families of MILPs:
given an infinite family of MILPs where each instance of the family has the same constraint matrix and the right-hand-sides of these instances belong to a lattice, we show that there exists a finite list of constraint aggregation multipliers that produces the GMIC closure for each instance in this family.
Second, inspired by this theoretical result, in Section~\ref{sec:heur} we propose a simple new heuristic for learning to generate constraint aggregation multipliers in order to approximate the GMIC closure.
In Section~\ref{sec:compute}, we conduct extensive computational experiments on randomly perturbed MIPLIB 2017~\cite{gleixner2021miplib} instances and show that the proposed heuristic can significantly accelerate the performance of Gurobi for many benchmark instances, even when taking into account the time required to predict aggregation multipliers and compute the cut coefficients.
To the best of our knowledge, this is the first work in the literature of data-driven cutting plane generation that is able to significantly accelerate the performance of a commercial state-of-the-art MILP solver, using default solver settings, on large-scale benchmark instances.
The proposed methods have been made publicly available as part of the open-source MIPLearn software package~\cite{xavier2020miplearn}.

 
\section{The Gomory Mixed-Integer Cut Closure}\label{sec:cutclosure}

Cutting-planes are an integral tool for improving the dual bounds in modern state-of-the-art MILP solvers~\cite{bixby2002solving,dey2018theoretical}.
One of the first classes of cutting-planes invented for MILPs was that of the Gomory Mixed-Integer Cuts (GMIC)~\cite{gomory1960algorithm,balas1996gomory,richard2010group,cornuejols2007revival}.
Given a single constraint,
\begin{eqnarray*}
S:=\left\{ x \in \mathbb{R}^{n}_{+}\,\left|\,\sum_{i \in [n]}\alpha_jx_j = \beta, \ x_j \in \mathbb{Z} \textup{ for }j \in J \right.\right\},
\end{eqnarray*}
where $J \subseteq [n],$ the GMIC is a valid inequality for $S$, which we denote as $\textup{GMIC}(S)$, and is obtained as:
\begin{eqnarray}\label{eq:GMIC}
\left.\begin{array}{l}\sum_{j \in J, f(\alpha_j) \leq f(\beta)}\frac{f(\alpha_j)}{f(\beta)}x_j + \sum_{i \in J, f(\alpha_j) > f(\beta)}\frac{1 - f(\alpha_j)}{1 -f(\beta)}x_j \\
+      
\sum_{j \in [n]\setminus J, \alpha_j \geq 0}\frac{\alpha_j}{f(\beta)}x_j +
\sum_{j \in [n]\setminus J, \alpha_j < 0}\frac{-\alpha_j}{1 - f(\beta)}x_j \geq 1, 
\end{array}\right\} \textup{GMIC}(S)
\end{eqnarray}
where $f(u) = u - \lfloor u \rfloor$.
Henceforth, for simplicity, we will also refer to the half-space described by the inequality (\ref{eq:GMIC}) as $\textup{GMIC}(S)$.
Moreover, if $\beta \in \mathbb{Z}$, then the inequality (\ref{eq:GMIC})  is not well-defined, in which case $\textup{GMIC}(S)$ will simply refer to $\mathbb{R}^n_{+}$. 
When $S$ is a row of the simplex tableau corresponding to basic integral variable, and the right-hand-side $\beta$ is fractional, then it is easy to see that the GMIC separates the current basis feasible solution.

A \emph{cutting-plane closure} is informally the set obtained by simultaneously adding all cuts that can be derived using a given type of cutting-plane procedure.
The practical importance of closures is closely related to the history of GMIC cuts.
In his seminal finite cutting-plane algorithm~\cite{gomory1963algorithm}, Gomory prescribes iteratively adding a \emph{single} cut that separates the current fractional optimal solution to the linear programming (LP) relaxation, and then re-resolving the updated linear program.
By the 1990s, several experimental studies using this \emph{single-cut} approach with GMICs, lead researchers, including Gomory himself, to mistakenly believe that GMIC were ``useless in practice"~\cite{cornuejols2007revival}.
One of the most exciting advances in the field of MILPs is a result by Balas et al.~\cite{balas1996gomory}, where they  debunked this belief by showing that GMICs perform very well in practice when inequalities are added \emph{simultaneously} from all fractional rows corresponding to integral basic variables of the optimal tableaux.
Although cut closures were already theoretically well-studied by the time of~\cite{balas1996gomory}, this work solidified their practical importance.

Given a general MILP in the following form:
\begin{eqnarray}\label{eq:standard}
\textup{IP}:=  \{x \in \mathbb{R}^{n}_+\,|\, Ax = b, \ x_j \in \mathbb{Z}\textup{ for } j \in J\},
\end{eqnarray}
where $A \in \mathbb{Q}^{m \times n}$ and $b \in \mathbb{Q}^m$, the GMIC closure is obtained by applying the GMIC procedure as described in (\ref{eq:GMIC}) to all single-constraint relaxations of $\textup{IP}$.
That is, the GMIC closure, denoted by $\mathcal{G}(\textup{IP})$, is given by:
\begin{eqnarray}\label{ex:GMIclosure}
\mathcal{G}(\textup{IP}):=\bigcap_{\lambda \in \mathbb{R}^m} \textup{GMIC}\left(\textup{IP}_\lambda \right),
\end{eqnarray}
where
\begin{eqnarray}\label{eq:lambda}
\textup{IP}_\lambda := \left\{x \in \mathbb{R}^{n}_+\,|\,\lambda^{\top}Ax = \lambda^{\top}b, \ x_j \in \mathbb{Z}\textup{ for } j \in J\right\}.   \end{eqnarray}
We will refer to $\lambda \in \mathbb{R}^m$ in (\ref{eq:lambda}) as an aggregation multiplier~\cite{bodur2018aggregation}.
All the cuts above are considered \emph{rank-1} as they can be directly generated from the original problem constraints.

The GMIC closure has various desirable theoretical and computational properties.
For example, the GMIC closure has interestingly been shown to be equivalent to various other cutting plane closures, including the mixed integer rounding inequality closure~\cite{nemhauser1990recursive,marchand2001aggregation,dash2010mir} and the split disjunctive cut closure~\cite{balas1979disjunctive,cook1990chvatal}, which highlights its theoretical importance.
Moreover, these closures have been shown to be polyhedral, with a finite number of the cuts dominating all the other cuts~\cite{cook1990chvatal,andersen2005split,vielma2007constructive,dash2010mir,averkov2012finitely}. 
Computationally, the strength of GMIC closures was empirically demonstrated through multiple studies on standard benchmark libraries~\cite{balas2008optimizing,dash2010mir}.
In particular, on MIPLIB 2003 instances~\cite{achterberg2006miplib}, Balas and Saxena~\cite{balas2008optimizing} showed that these cuts close more than 70\% of the integrality gap on average.
Another advantage of using rank-1 GMI cuts, in contrast to higher-rank cuts, is that they do not suffer significantly from the numerical challenges~\cite{cornuejols2013safety}. 

While the above theoretical and computational results make a compelling case for using the GMIC closure, in general this closure is very hard to build.
Indeed, separating an arbitrary point in the linear programming (LP) relaxation using a split cut (equivalent to GMI cut) is NP-hard~\cite{caprara2003separation,lee2019np}.
This challenge led to a number of research projects on approximating these closures through heuristic algorithms~\cite{dash2010heuristic,fischetti2013approximating}.
However, to the best of our knowledge, these heuristics remain too computationally expensive for most problems, and have not been integrated into commercial state-of-the-art MILP solvers.


\section{The GMIC Closure for an Infinite Family of MILPs}\label{sec:main}

In this paper, we choose to learn  aggregation multipliers $\lambda \in \mathbb{R}^m$ from historical data, which are then used to produce  one-constraint relaxations \eqref{eq:lambda} of the current instance under consideration.
We then apply \eqref{eq:GMIC} to these one-constraint relaxations to obtain GMICs for the current instance.

A basic question is the following: What ``information'' from historical data, if any, could be exploited to learn good aggregation multipliers?
For example, consider two integer programs:
$$
  \textup{max}\{c^{\top}x \,|\,Ax = b, x \in \mathbb{R}^n_{+}, x_j \in \mathbb{Z}, j \in J\},
$$
and its perturbation:
$$
  \textup{max}\{(c + \delta)^{\top}x \,|\,Ax = b + \varepsilon, x \in \mathbb{R}^n_{+}, x_j \in \mathbb{Z}, j \in J\},
$$
where $\delta \in \mathbb{R}^n$ and $\varepsilon \in \mathbb{R}^m$ are small perturbations.
If the LP relaxation of the first MILP is integral, then there are no interesting GMICs to be generated, that is there are no aggregation multipliers that lead to useful GMICs.
On the other hand, for arbitrarily small values of $\epsilon$, the LP relaxation of the second MILP may not be integral, and thus there may be several GMICs that are useful.
Therefore,  simple perturbation can change the GMIC closure dramatically. 

Somewhat counter-intuitively, the next theorem establishes that if we examine an \emph{infinite family of MILPs} with same constraint matrix and with right-hand-sides belonging to a lattice (so not arbitrarily close-by right-hand-sides) generated by rational vectors, then there is a \emph{finite list} of multipliers $\lambda$ that generate the GMIC closure of every instance. 

Consider the lattice $\Gamma \subseteq \mathbb{R}^m$ generated by rational vectors $b^1, \dots, b^k$, that is
$$\Gamma = \left\{\sum_{i = 1}^kz_i b^i\,|\, z_i \in \mathbb{Z}\ \forall i \in [k]\right\}.$$ 
For the rest of this section, we assume without loss of generality that all the data in (\ref{eq:standard}) is integral, that is $A \in \mathbb{Z}^{m \times n}$ and $b \in \mathbb{Z}^m$. 
We study the following parametric family of MILPs with varying right-hand-sides:
\begin{eqnarray}\label{eq:inffam}
\textup{IP}(\gamma):= \left\{ x\in \mathbb{R}^n_+\,\left|\, Ax = b + \gamma, \ x_j \in \mathbb{Z} \textup{ for }j \in J\right.\right\} \ \forall \gamma \in \Gamma.
\end{eqnarray}
With this notation, $\textup{IP}$ denoted in (\ref{eq:standard}) is equal to $\textup{IP}(0)$.

\begin{theorem}\label{thm:1}
Let $\Gamma$ be the lattice generated by rational vectors $b^1, \dots, b^k \in \mathbb{Q}^m$. Consider the infinite family of instances corresponding to $\Gamma$ as described in~(\ref{eq:inffam}).
Then 
there exists a finite set $\Lambda \subseteq \mathbb{R}^m$, such that the GMIC closure of every instance $\textup{IP}(\gamma)$ can be obtained using aggregation multipliers in $\Lambda$, that is:
$$\mathcal{G}(\textup{IP}(\gamma)) = \bigcap_{ \lambda \in \Lambda} \textup{GMIC}\left(\textup{IP}(\gamma)_\lambda\right)\quad \forall \gamma \in \Gamma.$$
\end{theorem}

\begin{remark}
    Dash et al.~\cite{dash2016polyhedrality} show that the split closure for a union of finite number of polyhedron can be described finitely.
    These results are not directly applicable in our case due to incompatible definitions.
    Another way to prove a result similar to Theorem~\ref{thm:1} is to consider a family of instances  where we apply unimodular transformations to the LP relaxation~\cite{dash2016polyhedrality}. 
    However, in this case we would be restricted to the right-hand-side being a lattice generated by the columns of the matrix $A$ and not an arbitrary lattice as presented in Theorem~\ref{thm:1}.
\end{remark}

In the case of pure integer programming, we can obtain a stronger result showing that there exist a finite set of multipliers (\ref{eq:lambda}) which gives the GMIC closure \emph{for all right-hand-sides.}

\begin{corollary}\label{cor:pureIP}
    Consider $\textup{IP}(\gamma)$ as defined in  (\ref{eq:inffam}), where we allow $\gamma \in \mathbb{R}^m$. Let 
    $J = [n]$, that is all variables are integral. Then there exists a finite set of multipliers $\Gamma$ such that 
    $$\mathcal{G}(\textup{IP}(\gamma)) = \bigcap_{ \lambda \in \Lambda} \textup{GMIC}\left(\textup{IP}(\gamma)_\lambda\right)\quad \forall \gamma \in \mathbb{R}^m.$$
\end{corollary}

\begin{remark}
Corollary~\ref{cor:pureIP} is similar in flavor to a result obtained by Wolsey~\cite{wolsey1981b} where it is shown that there is a finite list of subadditive functions that give the convex hull of pure integer programs for all right-hand-sides. 
\end{remark}

\section{Learning Heuristic}\label{sec:heur}
Theorem~\ref{thm:1}  suggests a simple learning heuristic.
When solving a distribution of MILP instances that have a fixed constraint matrix, one can sample multiple \emph{training instances} from this distribution, compute the aggregation multipliers that yield the GMIC closure for each training instance, then reuse these multipliers when solving new \emph{test instances} drawn from the distribution.

\paragraph{Cut Collection.} The first consideration in using this approach is determining how to generate aggregation multipliers for a given training sample --- we call this our \emph{cut collection algorithm}.
In this manuscript, we adopt a simplified version of the procedure by Fischetti and Salvagnin~\cite{fischetti2013approximating}, which generates GMIC cuts from multiple tableau bases.
Below, we provide a high-level summary of this method; for a more formal description and additional remarks, see Appendix~\ref{section:algorthm}.
The procedure begins by solving the LP relaxation of the problem and generating GMIC cuts for the optimal basis.
These cuts are then added as constraints into the LP relaxation, which is resolved, and dual values for the newly added constraints are computed.
Next, using these dual values as penalty, the cuts are transferred to the objective function of the LP relaxation, which is then resolved.
Finally, GMIC cuts are generated for this updated basis, and the process continues iteratively until a specified iteration limit is reached or a previously visited basis recurs.
We also have steps to eliminate cuts that became dominated as the algorithm progresses.

\paragraph{Cut Selection.}  A challenge in the design of cutting plane methods is deciding how many cuts to add.
It is well known that adding too many cuts can lead to poor algorithm performance~\cite{andreello2007embedding,dey2018theoretical}, not only due to the computational cost of generating cuts but also due to their impact on the size and density of the LP relaxation.
To achieve a better balance between gap closure and computational efficiency, we propose using multipliers collected only from a subset of training instances.
However, two questions remain:
how many training instances should be used, and how should they be selected?
To answer these questions, we conduct computational experiments with varying numbers of training instances, and we explore different selection strategies based on similarity: most-similar, most-different, and random instance selection.

\section{Computational Experiments}\label{sec:compute}

To evaluate the performance of the learning heuristic described in Section~\ref{sec:heur}, we implemented it in Julia/JuMP within the MIPLearn framework and conducted comprehensive experiments on randomly perturbed MIPLIB 2017 instances.

\subsection{Software and Hardware}

We utilized a single workstation computer (Ryzen 9 7950x, 16 cores, 32 threads, 128 GB DDR5) to generate all necessary training data and conduct all benchmarks.
During the training phase, we used Gurobi Optimizer 11.0.2 as our LP solver, whereas MIPLearn was used to convert problems into standard form, compute the tableau and generate GMI cuts.
In the testing phase, Gurobi acted as the MILP solver, receiving cuts computed by MIPLearn through a cut callback.
For all runs, Gurobi was configured to use a single thread, and 16 problems were solved in parallel at a time.
All other solver settings, including cut generation and pre-solve, were left at their default configurations.
To minimize solver performance variability, we used three random seeds and measured \emph{work units}, in addition to running time.
Additionally, since this manuscript focuses on the dual side of the solution process, Gurobi was provided with the optimal solution to all problems during testing.

\subsection{Benchmark Set Generation} 

To create realistic homogeneous families of MILP problems, we generated 55 variations (50 for training, 5 for testing) of each MIPLIB 2017 benchmark instance~\cite{gleixner2021miplib} by perturbing objective coefficients and right-hand-side values.
We outline the high-level approach below.
See Appendix~\ref{sec:InsPerturb} for details.

\paragraph{Right-Hand-Side (RHS) Changes.} 
If a constraint has at least one continuous variable or is an inequality, then we perturb the
right-hand-side value $b_i$ to $b_i \cdot r$ where $r$ is uniformly sampled from
$[0.9, 1.1]$.
If the constraint is an equality constraint with integer variables only, then we apply additive perturbation of $\{-1,0,1\}$.
Note that these perturbations may fail due to two reasons:
(i) none of the constraints satisfy the requirements;
(ii) the perturbed problem becomes infeasible.
If this happens (on any of the 55 generated variations), we disable RHS perturbation for this particular instance.

\paragraph{Objective Changes.}
We perturb each objective coefficient $c_j$ to $c_j\cdot r$ where $r$ is uniformly randomly selected  from $[0.75,1.25]$.
Note that, in some MIPLIB instances, the objective function contains a single variable, in which case the above perturbation does not produce any change.

\paragraph{Number of Perturbed Instances.}
Out of the 240 benchmark instances in MIPLIB 2017, we discarded 20 instances due to infeasibility, being very large, or being very difficult (requires more than $4$ hours to achieve an integer feasible solution). 
Out of the remaining 220 instances, we discovered that on 34 instances the perturbation rules failed to create different variations, leaving us with 186 instances where we were either successful in changing the RHS values (8 instances) or objective coefficients (105 instances) or both (73 instances).

\paragraph{Perturbation Quality.}
We solved the five test variations of the 186 instances, then observed optimal objective function values and solution diversity, focusing on the integer part of the optimal solution. 
Out of 186 instances, 178 showed at least two distinct optimal solutions (integer part). 
Notably, 157 instances had five distinct solutions, indicating that the rules were overall successful at generating instances with distinct optimal solutions and values. 
For only 8 instances, the rules produced variations with identical optimal solutions (integer part). 

\paragraph{Changes in Problem Difficulty.} 
A potential concern with the proposed perturbation rules was that the resulting instance variations might become easier or trivial to solve.
On the contrary we discovered that the variations, on average, became more difficult to solve. 
In particular, the arithmetic (geometric resp.) mean of solution times increased from 911 seconds to 1868 seconds (174 seconds to 294 seconds resp.) and similar observations were made for the number of nodes in the branch-and-bound tree.

\subsection{Expert Method Performance}\label{sec:perexpert}

Before benchmarking the performance of the learning heuristic, we begin by evaluating a theoretical \emph{expert method} to gauge what potential speedups are achievable with this approach and to identify which MIPLIB 2017 instances could potentially benefit.
Given a \emph{test} instance, the \emph{expert} method first runs the cut collection algorithm to generate GMI cuts, then invokes Gurobi to solve the instance, providing the cuts to the solver though a cut callback.
Specifically, all collected cuts are simultaneously provided exactly once, when the cut callback is called for the first time.
However, the time spent collecting the cuts is not counted towards the running time of the expert method, to simulate the ideal results that a ``perfect'' learning heuristic could achieve.

\paragraph{Eliminated Instances.} While benchmarking the expert method, 48 additional were eliminated, mostly due to the cut collection procedure returning zero cuts, but also due to some memory/time limits and a few numerical errors.
See Table~\ref{tab:collectfail} for a more detailed breakdown.
Zero cuts are typically generated due to dual degeneracy, as it is possible that there is no improvement in dual bound after a single round of cuts and the dual multipliers corresponding to the cuts are then zero. 
In addition to the instances listed in Table~\ref{tab:collectfail}, on three instances we observed numerical issues during the benchmark procedure after cut collection.

\paragraph{Cut Performance within Gurobi.}
On the remaining 135 instances, we evaluate the effectiveness of the cuts based on Gurobi's runtime, the number of branch-and-bound nodes, and work units, a deterministic measure reflecting time spent on optimization.
For each of the five test variations and three random seeds, we solved each instance using Gurobi both with and without cuts, gathering average statistics.
To classify instances with respect to cut performance, we calculated speedup as the ration between average work units with and without cuts.
Instances were labeled \emph{positive} if the average speedup was above 1.01x, \emph{negative} if below 0.99x, and \emph{neutral} otherwise.
Our results showed that 50 instances were positive, 36 were neutral and 49 were negative.
See Table~\ref{tab:expert} in the Appendix for more detailed results.
These findings suggest that, \emph{if we could ignore the computational cost of generating them}, these cuts could improve solver performance for a substantial number of MIPLIB 2017 instances.
However, when accounting for the time required to collect the cuts, most positive instances become negative, which motivates our learning approach.

\subsection{Learning Heuristic Performance}\label{sec:ml}

In this subsection, we evaluate the performance of the learning heuristic, focusing on the 50 benchmark instances in which the \emph{expert} provided positive results.

\paragraph{Learning Setting.}
Recall that, for each of the 50 benchmark instances considered, we have 50 training and 5 test variations.
During an offline training phase, we ran the cut collection method on all training variations and recorded the generated cuts.
More specifically, for each generated cut, we stored the basis and corresponding tableau row that was used to generate the cut, which can be seen as an alternative method of storing the necessary constraint aggregation multipliers.
During the test phase, as before, the cuts were generated and provided to Gurobi via a cut callback function.
Also, as before, each problem was solved three times with different random seeds and average statistics were collected. 

\paragraph{Evaluated Methods.}

As discussed in Section~\ref{sec:heur}, we would like to test two aspects of our learning heuristics.
First, the number of cuts to add, which, in our case, depends primarily on the number $k$ of training variations from which to select aggregation multipliers.
To cover a wide range of settings, while keeping the time required to run the computational experiments manageable, we chose $k = \{1, 10, 50\}$ variations.
Second, the strategy for selecting training variations.
Specifically, we considered:
(i) $k$-closest training variations to the test instance, (ii) $k$-farthest variations, and (iii) $k$-random variations.
Based on these criteria, we tested the following configurations: (i) 1-closest, (ii) 10-closest, (iii) 10-farthest, (iv) 10-random, and (v) 50-closest, effectively using all training variations to select aggregation multipliers.
To measure instance similarity, each instance variation was represented by a feature vector consisting of right-hand-side values and objective coefficients and distances were measured using the 2-norm.
Since these features can vary significantly in magnitude, we applied standard scaling by removing the mean and scaling to unit variance.

\begin{table}[tbp]
  \centering
  \caption{Average performance over all 50 benchmark instances.}
 \include{6_ml_overview_all1}
  \label{tab:ml-overview-all}%
\end{table}%

\begin{table}[tbp]
  \centering
  \caption{Average performance over 14 hard benchmark instances.}
  \include{6_ml_overview_hard1}
  \label{tab:ml-overview-hard}%
\end{table}%

Table~\ref{tab:ml-overview-all} summarizes the average performance of the evaluated methods across all 50 benchmark instances.
For each method, \emph{time} includes not only the time Gurobi requires to solve the problem, but also the additional time spent predicting multipliers, converting the test problem to standard form, computing multiple tableaux, and generating cuts before the optimization process begins.
In contrast, \emph{work units} measures only the computational effort spent by Gurobi, in a deterministic way.
The table also includes the average number of branch-and-bound nodes explored and average speedups for both time and work units.
Speedups were calculated by first computing them for each test variation individually, then averaging these results across variations.
The \emph{baseline} row in the table refers to standard Gurobi, with an empty cut callback function.
Additionally, since the \emph{expert}'s time excludes cut collection time, we introduce an \emph{exp+col} row, which adds cut collection time to the \emph{time} column while keeping all other entries unchanged.
Table~\ref{tab:ml-overview-hard} shows similar results, but focuses on 14 hard benchmark instances, defined as those requiring at least five minutes for default Gurobi to solve.
More detailed results are presented in Tables~\ref{tab:mlTime1}, \ref{tab:mlTime2}, \ref{tab:mlWork1}, and \ref{tab:mlWork2}, in the Appendix.
We can make the following observations:
\begin{enumerate}
    \item \emph{All variations of the learning heuristic are quite effective at improving Gurobi's running time} (measured in work units), with the best variant achieving an average speedup of 1.49x across all 50 instances and a 1.40x speedup on the 14 hard instances.
    Every variation, including the simple 1-closest, outperformed Gurobi's default, and most surpassed even the expert method!
    With few exceptions, when the expert method performs well, the learning heuristics achieve similarly strong results.
    \item \emph{More cuts are typically better} (with respect to work units).
    Although there is typically a concern that providing too many cuts to a MILP solver may hinder its performance, we did not observe this issue in our experiments.
    In fact, the best performance was obtained by 50-closest, which adds cuts based on aggregation multipliers from all the training instances.
    \item \emph{Instance similarity does not appear to be a critical factor}.
    All three selection strategies (closest, farthest and random) showed similar performance, as long as cuts were generated from 10 training variations.
    If anything, 10-farthest performed the best in our experiments among these three strategies.
    \item \emph{The heuristic performs well even when right-hand-side (RHS) values are perturbed}.
    When the RHS remains unchanged, one may expect the learning heuristics to perform well, as the feasible region of the problems remain the same.
    However, the real test is how the heuristic handles varying RHS values.
    Of the 50 instances tested, 20 included RHS perturbations, and we observed similar speedup trends in these 20 instances as in the full set of 50.
    Specifically, the arithmetic average of work units for baseline, expert, $10$-farthest, $1$-closest, $10$-closest, $50$-closest, and $10$-random is $2535.55$, $	2414.54$, $2242.60$, $2312.10$, $2343.43$, $2177.83$, and $2327.89$ respectively. 
    \item \emph{The learning heuristic is effective even when considering total running time, particularly for challenging instances.}
    The previous discussion points focus on work units, which measure only the effort spent by Gurobi.
    When we also include the computational time required for predicting and generating cuts, some new insights emerge.
    First, adding cut collection time to the expert method completely nullifies its running time improvements, with an average speedup of just 0.46x across all 50 benchmark instances.
    Even for hard instances, the \emph{exp+col} method still underperforms the baseline, with an average speedup of 0.91x.
    Second, although the 50-closest variant provides the best results in terms of work units, it is less efficient in total running time than other methods, due to the extensive data processing it requires, such as reading additional training data files and computing tableaux for more linear programming bases.
    Methods with lower preprocessing demands, like 10-farthest and 10-random, offer a more balanced performance.
    Nonetheless, all learning heuristic variants outperform the baseline on hard instances, with time speedups between 1.08x and 1.17x.
    Third, our cut generation implementation still suffers from significant overhead, limiting the effectiveness of methods like 10-farthest to hard instances only.
    In fact, the only method able to achieve a positive average time speedup across all 50 benchmark instances is 1-closest, which requires the least preprocessing.
    Much of this overhead comes from recomputing the problem's standard form and multiple tableaux, as we lack access to solver internals containing this information.
    If these methods were directly integrated into Gurobi, we anticipate much lower overhead, making the method potentially useful even for easy instances.
\end{enumerate}

\section{Conclusion and Future Research}\label{sec:con}
In this work, we demonstrated that historical data can be effectively used to generate strong cutting planes and accelerate the performance of a commercial state-of-the-art MILP solver.
There are several open theoretical and computational research directions. 
On the theory side, there are questions regarding generalizing Theorem~\ref{thm:1} (and similar result for split cuts with the same proof; see Theorem~\ref{thm:split} in Appendix). We have established finite list of aggregation multipliers only when the right-hand-side of the instance family lives in a lattice generated by rational vectors. 
What happens when we consider all right-hand-sides is an open question.
Finally, examining extensions of such results for more generals families of cut from the simplex tableau, the so-called group cuts ~\cite{richard2010group} is another interesting direction of research.

On the computational side, we have seen that it is possible to get substantial improvements over default Gurobi, just by re-using the good aggregation multipliers determined for  historical instances. 
This results could perhaps be improved by examining  some interesting extensions. 
First, with more computational power and effort, it may be possible to work with a better expert, which for example generates cuts even from infeasible basis. 
Second, in our study we have explored changing the right-side within a margin of $\pm 10\%$ and the objective coefficient by $\pm 25\%$.
Perhaps, it is possible to explore more dramatic changes. 

\section*{Acknowledgments}

Argonne National Laboratory's work was supported by the U.S. Department of Energy, Office of Electricity, under contract DE-AC02-06CH11357.
Santanu S. Dey would like to gratefully acknowledge the support of ONR grant \#N000142212632.

\bibliographystyle{plain}
\bibliography{ref}

\appendix
\section{Proofs}

Given a vector $u \in \mathbb{R}^k$, let $u^+, u^- \in \mathbb{R}^k_+$ be the vectors defined as: $$u^+_i = \left\{\begin{array}{rl} u_i & \textup{if }u_i \geq 0 \\
  0 &  \textup{if }u_i \leq 0.
  \end{array}\right. \quad u^-_i = \left\{\begin{array}{rl} -u_i & \textup{if }u_i < 0 \\
  0 &  \textup{if }u_i \geq 0.
  \end{array}\right. $$
  Given two inequalities 
  \begin{eqnarray}\label{eq:com1}
  \eta^1x \geq 1,
  \end{eqnarray}
  and 
  \begin{eqnarray}\label{eq:com2}
  \eta^2x \geq 1
  \end{eqnarray}
  valid for \textup{IP} presented in (\ref{eq:standard}), we say that 
  (\ref{eq:com1}) dominates (\ref{eq:com2}) if $\eta^1 \leq \eta^2$. This is because all the variables in $\textup{IP}$ are non-negative. 
  
As discussed in Section~\ref{sec:cutclosure}, split cuts and GMI cuts are very closely related. Our presentation in this section is from Section 5.1 in~\cite{conforti2014integer}. 

\paragraph{Split cuts for $\textup{IP}$.}
Considering a mixed integer set $\textup{IP}$ in standard form as described in (\ref{eq:standard}), any non-dominated split cut for $\textup{IP}$ can be generated as follows. Given a vector $(\lambda, v) \in \mathbb{R}^m \times \mathbb{R}^n$, we will say it is \emph{legitimate} if it satisfies the following three conditions:
\begin{enumerate}
    \item[(1.)] \label{cond1}$\lambda^{\top}A_j - v_j \in \mathbb{Z}$ for $j \in J$,
    \item[(2.)] \label{cond2}$\lambda^{\top}A_j - v_j = 0$ for $j \in [n]\setminus J$,  \item[(3.)]\label{cond3} $\lambda^{\top}b \not\in \mathbb{Z},$ where $f= \lambda^{\top}b - \lfloor \lambda^{\top}b \rfloor.$ 
\end{enumerate}
For a legitimate $(\lambda, v) \in \mathbb{R}^m \times \mathbb{R}^n$, the inequality
\begin{eqnarray}\label{eq:split}
\frac{v^{+}x}{f} + \frac{v^{-}x}{1-f} \geq 1,
\end{eqnarray}
is a split inequality valid for $\textup{IP}$ defined in (\ref{eq:standard}) obtained using the disjunction $$(\pi^{\top} x \leq \pi_0)\vee (\pi^{\top}x \geq \pi_0 +1),$$ where $\pi= \lambda^{\top}A - v$ and $\pi_0 = \lfloor \lambda^{\top}b\rfloor$. 
\paragraph{Connection to GMICs.}
Using the fact that variables in $\textup{IP}$ given in (\ref{eq:standard}) are non-negative, one can easily show the following: (Section 5.1.4 in ~\cite{conforti2014integer}) Given a fixed $\hat{\lambda} \in \mathbb{R}^m$ with $\hat{\lambda}^{\top}b \notin \mathbb{Z}$, one can find one optimal value of $\hat{v}$ such that $(\hat{\lambda}, \hat{v})$ is legitimate, where optimality implies that an inequality of the form (\ref{eq:split})
corresponding to any legitimate pair $(\hat{\lambda}, v)$ is dominated by the inequality (\ref{eq:split})
corresponding to $(\hat{\lambda}, \hat{v})$. This inequality corresponding to $(\hat{\lambda}, \hat{v})$ is precisely the GMIC as described in (\ref{eq:GMIC}) with $\alpha_j = \hat{\lambda}^{\top}A_j$ for $j \in [n]$ and $\beta = \hat{\lambda}^{\top}b$, that is the inequality $GMIC(\textup{IP}_{\hat{\lambda}})$. Thus, all the rank-1 GMICs for the set $\textup{IP}$ are special case of (\ref{eq:split}), where the $v$ vectors is selected optimally for a given $\lambda$.  

We next show that by a  argument similar to that presented in~\cite{dash2010mir} (our presentation is based on~\cite{conforti2014integer}), it is possible to prove that there is a \emph{finite set of aggregation multipliers}
which yield the GMI closure for an infinite family of instances  with varying right-hand-sides. In order to present this result we need the notation presented below.


We begin with a result presented in~\cite{conforti2014integer}, originally shown in~\cite{andersen2005split}, before presenting our main result.
\begin{lemma}\label{lemma:1} (Corollary 5.6 in~\cite{conforti2014integer})
If the inequality (\ref{eq:split}) corresponding to a  $(\lambda, v)$ is  undominated among all possible inequalities of the form (\ref{eq:split}), then the support of $(\lambda, v)$ corresponds to a basis, that is, if $\lambda_{i_1}, \dots, \lambda_{i_k}$ are non-zero and $v_{j_1}, \dots, v_{j_l}$ are non-zero, then $\{a_{i_1}, \dots, a_{i_k}, e_{j_1}, \dots, e_{j_l}\} \in \mathbb{R}^n$ are linearly independent, where $a_{i}$ is the $i^{th}$ row of $A$ and $e_j$ is the $j^{\textup{th}}$ standard basis vector.
\end{lemma}

\begin{theorem}
Let $\Gamma$ be the lattice generated by rational vectors $b^1, \dots, b^k \in \mathbb{Q}^m$. Consider the infinite family of instances corresponding to $\Gamma$ as described in~(\ref{eq:inffam}).
Then 
there exists a finite set $\Lambda \subseteq \mathbb{R}^m$, such that the GMIC closure of every instance $\textup{IP}(\gamma)$ can be obtained using aggregation multipliers in $\Lambda$, that is:
$$\mathcal{G}(\textup{IP}(\gamma)) = \bigcap_{ \lambda \in \Lambda} \textup{GMIC}\left(\textup{IP}(\gamma)_\lambda\right)\quad \forall \gamma \in \Gamma.$$
\end{theorem}
\begin{proof}
We will show that there exists a finite set  $\tilde{\Lambda} = \{(\lambda^1, v^1), \dots, (\lambda^g, v^g)\}$ with its elements satisfying the first two conditions of being legitimate (if third is not satisfied for some $b + \gamma$, then no cut using this $\lambda$ is generated for this instance) and satisfying Lemma~\ref{lemma:1}, 
such that 
any inequality (\ref{eq:split}) generated from any legitimate $(\lambda, v)$
is dominated by an inequality (\ref{eq:split}) corresponding to $(\lambda^i, v^i)$ for some $i \in [g]$. 
Let the projection of $\tilde{\Lambda}$ onto the $\lambda$-component be the finite set $\Lambda$. As discussed above, given a fixed $\hat{\lambda}$, we can then select a corresponding $\hat{v}$ to this $\hat{\lambda}$, which leads to an inequality of the type (\ref{eq:split}) that dominates all other inequalities of the type (\ref{eq:split}) corresponding to legitimate vector $(\hat{\lambda}, v)$  -- this is how all GMICs are generated. Thus, this shows that the GMIC closure is obtained by generating the GMICs corresponding to the elements in the finite list $\Lambda$. 

Let $u = (\lambda, v)$ be a legitimate vector with a support corresponding to a basis (from Lemma~\ref{lemma:1})
for an instance $\textup{IP}(\gamma)$, that is it satisfies (1.), (2.) and $\lambda^{\top}(b + \gamma) \not\in \mathbb{Z}$. 

We define the cone 
\begin{eqnarray*}
\mathcal{C}_{\textup{sign}(u)} = \left\{w = (\theta, \phi)\in \mathbb{R}^{m}\times\mathbb{R}^n \,\left|\,\begin{array}{rcl} \theta^{\top}A_j - \phi_j &=& 0 \ \forall j \in [n]\setminus J \\
w_j &=& 0 \textup{ if }u_j = 0 \\
w_j & \leq& 0 \textup{ if }u_j < 0 \\
w_j & \geq & 0 \textup{ if }u_j > 0.
\end{array}\right.\right\}.
\end{eqnarray*}
We claim that if there exists $u^1 = (\theta^1, \phi^1), u^2 = (\theta^2, \phi^2) \in \mathcal{C}_{\textup{sign}(u)}$ such that 
$u = u^1 + u^2$, where $u^2 \in \mathbb{Z}^{n + m}$ and $(\theta^{2})^{\top}(b + \gamma) \in \mathbb{Z}$, then the inequality corresponding to $u$ is dominated by the inequality corresponding to $u^1$. First we verify that $u^1$ is legitimate that is it satisfies conditions (1.) - (3.). Note that since $A $ is an integral matrix, and $(\theta^2, \phi^2_j) \in \mathbb{Z}^m \times \mathbb{Z}$, for $j \in J$, we have that $\mathbb{Z} \ni \lambda^{\top}A_j - v_j = (\theta^1 + \theta^2)^{\top}A_j - \phi^1_j - \phi^2_2$ implies that $(\phi^1)^{\top}A_j - v_j \in \mathbb{Z},$ that is condition (1.) holds. Also because $u^1 \in C_u$, we have that $u^1$ satisfies  condition (2.). Finally, note that $\lambda^{\top}(b + \gamma) = (\theta^1 + \theta^2)^{\top}(b + \gamma)$ and by assumption $(\theta^2)^{\top}(b + \gamma) \in \mathbb{Z}$, we have $(\theta^1)^{\top}(b + \gamma) \equiv \lambda^{\top}(b + \gamma) (\textup{mod} 1) \equiv f (\textup{mod} 1)$, that is $ (\theta^1)^{\top}(b + \gamma) \not\in \mathbb{Z}$, thus $u^1$ satisfies condition (3.). Finally note that due to constraints defining $\mathcal{C}_{\textup{sign}(u)}$, and because $u = u^1 + u^2$, we have that 
$(\phi^1)^+ \leq v^+$ and $(\phi^1)^- \leq v^-$.
Therefore, the inequality corresponding to $u^1$:
$$\frac{((\phi^1)^+)^{\top}x}{f}+ \frac{((\phi^1)^-)^{\top}x}{1-f}\geq 1,$$
dominates the inequality corresponding to $u = (\lambda, v)$.

Let $b^1, \dots, b^k$ be the generators of $\Gamma$. Let  $M(\Gamma) \in \mathbb{Z}_{+}$ be the smallest positive integer such that $M(\Gamma)\cdot b^i \in \mathbb{Z}^m$ for all $i \in [k].$ Note that this implies that $M(\Gamma)\cdot \gamma \in \mathbb{Z}^m $ for all $\gamma \in \Gamma.$ Let 
$\Delta$ be the maximum among the absolute value of all sub-determinant of the matrix corresponding to the left-hand-side of constraints defining $\mathcal{C}_{\textup{sign}(u)}$.
Let $H:= (m+ n)\cdot M(\Gamma)\cdot \Delta$.

We next show that for all legitimate $u \in \mathbb{R}^{m+n}$ there exists 
$u^1,u^2 \in \mathcal{C}_{\textup{sign}(u)}$ such that $u=u^1+u^2$ and $u^2 \in \mathbb{Z}^{m+n}$, $(\theta^{2})^{\top}(b + \gamma) \in \mathbb{Z}$ for all $\gamma \in \Gamma$, 
where $u^1$ satisfies $-H\leq u^1_j \leq H$ for all $j \in [m+ n]$ and $u^1$s support is contained in the support of $u$. The key to this proof depends on the fact that $\mathcal{C}_{\textup{sign}(u)}$ does not depend on the right-hand-side and that $M(\Gamma)\cdot \gamma \in \mathbb{Z}^m $ for all $\gamma \in \Gamma.$ First note that $\mathcal{C}_{\textup{sign}(u)}$ is a pointed cone and let $r^1, \dots, r^k$ be the extreme rays of $\mathcal{C}_{\textup{sign}(u)}$. By standard arguments (see Chapter 3 in~\cite{schrijver1998theory}) we can re-scale $r^1, \dots, r^k$ such that they are integral vectors and $\|r^t\|_{\infty} \leq \Delta$ for all $t \in [k]$. Let us now further re-scale these vectors by $M(\Gamma),$ that is we assume:
\begin{itemize}
\item $r^t \in \mathbb{Z}^{m+n}$ for all $t \in [k]$,
\item Each entry of $r^t$ is divisible by $M(\Gamma)$ for all $t \in [k]$,
\item $\|r^t\|_{\infty} \leq M(\Gamma)\cdot\Delta$.
\end{itemize}
Now since $u \in \mathcal{C}_{\textup{sign}(u)}$, we have that $u = \sum_{t = 1}^k \psi_tr^t$, where by Carath\'eodory's Theorem at most $m + n $ of the $\psi$s are positive. Let $u^1 = \sum_{t = 1}^k (\psi_t - \lfloor \psi_t\rfloor)\cdot r^t$ and let $u^2 = \sum_{t = 1}^k \lfloor \psi_t\rfloor\cdot r^t$. Then clearly $u^1, u^2 \in \mathcal{C}_{\textup{sign}(u)}$, $u^2 \in \mathbb{Z}^{m+n}$ and each entry of $u^2$ is divisible by $M(\Gamma)$ for all $\gamma \in \Gamma$. Thus $(\theta^2)^{\top}(b + \gamma) \in \mathbb{Z}$. Finally, since at most $n+m$ of $\psi_t$s are positive, we have that $\|u^1\|_{\infty} \leq (m+n)\cdot M(\gamma)\cdot\Delta = H$.

Based on the results above, we conclude that non-dominated inequalities (\ref{eq:split}) for all instances $\textup{IP}(\gamma)$ for $\gamma \in \Gamma$ are generated corresponding to legitimate vectors $u = (\lambda, v)$ with a support corresponding to a basis and $\|u\|_{\infty} \leq H.$ Therefore, we obtain that the set:
$$\Pi:= \{ \pi \in \mathbb{Z}^n\,|\, \pi = \lambda^{\top}A - v^{\top}I, \pi_j = 0 \ \forall j \in [n]\setminus J, \|\lambda\|_{\infty}\leq H, \|v\|_{\infty}\leq H  \},$$
is finite. Thus, all the non-dominated split cuts (\ref{eq:split}) can be generated using the split disjunctions whose left-hand-side belong to the $\Pi$. 

Finally, note that for a given $\pi \in \Pi$, the solution to the system $\pi = \lambda^{\top}A - u^{\top}I$ where the support of $(\lambda, u)$ corresponds to a basis is a unique solution. Thus the number of non-dominated inequalities of the form (\ref{eq:split}) can be obtained from a finite list of $u = (\lambda, v)$s. This completes the proof. 
\end{proof}

The proof above closely follows arguments presented in~\cite{conforti2014integer}.

\begin{theorem}
Let $\Gamma$ be the lattice generated by rational vectors $b^1, \dots, b^k \in \mathbb{Q}^m$. Consider the infinite family of instances corresponding to $\Gamma$ as described in~(\ref{eq:inffam}).
Then 
there exists a finite set $\Lambda \subseteq \mathbb{R}^m$, such that the GMIC closure of every instance $\textup{IP}(\gamma)$ can be obtained using aggregation multipliers in $\Lambda$, that is:
$$\mathcal{G}(\textup{IP}(\gamma)) = \bigcap_{ \lambda \in \Lambda} \textup{GMIC}\left(\textup{IP}(\gamma)_\lambda\right)\quad \forall \gamma \in \Gamma.$$
\end{theorem}
\begin{proof}
We will show that there exists a finite set  $\tilde{\Lambda} = \{(\lambda^1, v^1), \dots, (\lambda^g, v^g)\}$ with its elements satisfying the first two conditions of being legitimate (if third is not satisfied for some $b + \gamma$, then no cut using this $\lambda$ is generated for this instance) and satisfying Lemma~\ref{lemma:1}, 
such that 
any inequality (\ref{eq:split}) generated from any legitimate $(\lambda, v)$
is dominated by an inequality (\ref{eq:split}) corresponding to $(\lambda^i, v^i)$ for some $i \in [g]$. 
Let the projection of $\tilde{\Lambda}$ onto the $\lambda$-component be the finite set $\Lambda$. As discussed above, given a fixed $\hat{\lambda}$, we can then select a corresponding $\hat{v}$ to this $\hat{\lambda}$, which leads to an inequality of the type (\ref{eq:split}) that dominates all other inequalities of the type (\ref{eq:split}) corresponding to legitimate vector $(\hat{\lambda}, v)$  -- this is how all GMICs are generated. Thus, this shows that the GMIC closure is obtained by generating the GMICs corresponding to the elements in the finite list $\Lambda$. 

Let $u = (\lambda, v)$ be a legitimate vector with a support corresponding to a basis (from Lemma~\ref{lemma:1})
for an instance $\textup{IP}(\gamma)$, that is it satisfies (1.), (2.) and $\lambda^{\top}(b + \gamma) \not\in \mathbb{Z}$. 

We define the cone 
\begin{eqnarray*}
\mathcal{C}_{\textup{sign}(u)} = \left\{w = (\theta, \phi)\in \mathbb{R}^{m}\times\mathbb{R}^n \,\left|\,\begin{array}{rcl} \theta^{\top}A_j - \phi_j &=& 0 \ \forall j \in [n]\setminus J \\
w_j &=& 0 \textup{ if }u_j = 0 \\
w_j & \leq& 0 \textup{ if }u_j < 0 \\
w_j & \geq & 0 \textup{ if }u_j > 0.
\end{array}\right.\right\}.
\end{eqnarray*}
We claim that if there exists $u^1 = (\theta^1, \phi^1), u^2 = (\theta^2, \phi^2) \in \mathcal{C}_{\textup{sign}(u)}$ such that 
$u = u^1 + u^2$, where $u^2 \in \mathbb{Z}^{n + m}$ and $(\theta^{2})^{\top}(b + \gamma) \in \mathbb{Z}$, then the inequality corresponding to $u$ is dominated by the inequality corresponding to $u^1$. First we verify that $u^1$ is legitimate that is it satisfies conditions (1.) - (3.). Note that since $A $ is an integral matrix, and $(\theta^2, \phi^2_j) \in \mathbb{Z}^m \times \mathbb{Z}$, for $j \in J$, we have that $\mathbb{Z} \ni \lambda^{\top}A_j - v_j = (\theta^1 + \theta^2)^{\top}A_j - \phi^1_j - \phi^2_2$ implies that $(\phi^1)^{\top}A_j - v_j \in \mathbb{Z},$ that is condition (1.) holds. Also because $u^1 \in C_u$, we have that $u^1$ satisfies  condition (2.). Finally, note that $\lambda^{\top}(b + \gamma) = (\theta^1 + \theta^2)^{\top}(b + \gamma)$ and by assumption $(\theta^2)^{\top}(b + \gamma) \in \mathbb{Z}$, we have $(\theta^1)^{\top}(b + \gamma) \equiv \lambda^{\top}(b + \gamma) (\textup{mod} 1) \equiv f (\textup{mod} 1)$, that is $ (\theta^1)^{\top}(b + \gamma) \not\in \mathbb{Z}$, thus $u^1$ satisfies condition (3.). Finally note that due to constraints defining $\mathcal{C}_{\textup{sign}(u)}$, and because $u = u^1 + u^2$, we have that 
$(\phi^1)^+ \leq v^+$ and $(\phi^1)^- \leq v^-$.
Therefore, the inequality corresponding to $u^1$:
$$\frac{((\phi^1)^+)^{\top}x}{f}+ \frac{((\phi^1)^-)^{\top}x}{1-f}\geq 1,$$
dominates the inequality corresponding to $u = (\lambda, v)$.

Let $b^1, \dots, b^k$ be the generators of $\Gamma$. Let  $M(\Gamma) \in \mathbb{Z}_{+}$ be the smallest positive integer such that $M(\Gamma)\cdot b^i \in \mathbb{Z}^m$ for all $i \in [k].$ Note that this implies that $M(\Gamma)\cdot \gamma \in \mathbb{Z}^m $ for all $\gamma \in \Gamma.$ Let 
$\Delta$ be the maximum among the absolute value of all sub-determinant of the matrix corresponding to the left-hand-side of constraints defining $\mathcal{C}_{\textup{sign}(u)}$.
Let $H:= (m+ n)\cdot M(\Gamma)\cdot \Delta$.

We next show that for all legitimate $u \in \mathbb{R}^{m+n}$ there exists 
$u^1,u^2 \in \mathcal{C}_{\textup{sign}(u)}$ such that $u=u^1+u^2$ and $u^2 \in \mathbb{Z}^{m+n}$, $(\theta^{2})^{\top}(b + \gamma) \in \mathbb{Z}$ for all $\gamma \in \Gamma$, 
where $u^1$ satisfies $-H\leq u^1_j \leq H$ for all $j \in [m+ n]$ and $u^1$s support is contained in the support of $u$. The key to this proof depends on the fact that $\mathcal{C}_{\textup{sign}(u)}$ does not depend on the right-hand-side and that $M(\Gamma)\cdot \gamma \in \mathbb{Z}^m $ for all $\gamma \in \Gamma.$ First note that $\mathcal{C}_{\textup{sign}(u)}$ is a pointed cone and let $r^1, \dots, r^k$ be the extreme rays of $\mathcal{C}_{\textup{sign}(u)}$. By standard arguments (see Chapter 3 in~\cite{schrijver1998theory}) we can re-scale $r^1, \dots, r^k$ such that they are integral vectors and $\|r^t\|_{\infty} \leq \Delta$ for all $t \in [k]$. Let us now further re-scale these vectors by $M(\Gamma),$ that is we assume:
\begin{itemize}
\item $r^t \in \mathbb{Z}^{m+n}$ for all $t \in [k]$,
\item Each entry of $r^t$ is divisible by $M(\Gamma)$ for all $t \in [k]$,
\item $\|r^t\|_{\infty} \leq M(\Gamma)\cdot\Delta$.
\end{itemize}
Now since $u \in \mathcal{C}_{\textup{sign}(u)}$, we have that $u = \sum_{t = 1}^k \psi_tr^t$, where by Carath\'eodory's Theorem at most $m + n $ of the $\psi$s are positive. Let $u^1 = \sum_{t = 1}^k (\psi_t - \lfloor \psi_t\rfloor)\cdot r^t$ and let $u^2 = \sum_{t = 1}^k \lfloor \psi_t\rfloor\cdot r^t$. Then clearly $u^1, u^2 \in \mathcal{C}_{\textup{sign}(u)}$, $u^2 \in \mathbb{Z}^{m+n}$ and each entry of $u^2$ is divisible by $M(\Gamma)$ for all $\gamma \in \Gamma$. Thus $(\theta^2)^{\top}(b + \gamma) \in \mathbb{Z}$. Finally, since at most $n+m$ of $\psi_t$s are positive, we have that $\|u^1\|_{\infty} \leq (m+n)\cdot M(\gamma)\cdot\Delta = H$.

Based on the results above, we conclude that non-dominated inequalities (\ref{eq:split}) for all instances $\textup{IP}(\gamma)$ for $\gamma \in \Gamma$ are generated corresponding to legitimate vectors $u = (\lambda, v)$ with a support corresponding to a basis and $\|u\|_{\infty} \leq H.$ Therefore, we obtain that the set:
$$\Pi:= \{ \pi \in \mathbb{Z}^n\,|\, \pi = \lambda^{\top}A - v^{\top}I, \pi_j = 0 \ \forall j \in [n]\setminus J, \|\lambda\|_{\infty}\leq H, \|v\|_{\infty}\leq H  \},$$
is finite. Thus, all the non-dominated split cuts (\ref{eq:split}) can be generated using the split disjunctions whose left-hand-side belong to the $\Pi$. 

Finally, note that for a given $\pi \in \Pi$, the solution to the system $\pi = \lambda^{\top}A - u^{\top}I$ where the support of $(\lambda, u)$ corresponds to a basis is a unique solution. Thus the number of non-dominated inequalities of the form (\ref{eq:split}) can be obtained from a finite list of $u = (\lambda, v)$s. This completes the proof. 
\end{proof}

The proof above closely follows arguments presented in~\cite{conforti2014integer}.

\begin{corollary}
  Consider $\textup{IP}(\gamma)$ as defined in  (\ref{eq:inffam}), where we allow $\gamma \in \mathbb{R}^m$. Let 
  $J = [n]$, that is all variables are integral. Then there exists a finite set of multipliers $\Gamma$ such that 
  $$\mathcal{G}(\textup{IP}(\gamma)) = \bigcap_{ \lambda \in \Lambda} \textup{GMIC}\left(\textup{IP}(\gamma)_\lambda\right)\quad \forall \gamma \in \mathbb{R}^m.$$
\end{corollary}
\begin{proof}
  Let $\Delta^1$ be the finite set obtained from Theorem~\ref{thm:1} where $\Gamma = \mathbb{Z}^m$, that is $\Gamma$ is generated by $e^1, e^2, \dots, e^m$. 
  Also let $\Delta^2 := \{ e^1, e^2, \dots, e^m\}.$ Let $\Delta = \Delta^1 \cup \Delta^2$. We consider two cases:
  \begin{itemize}
  \item $b_i + \gamma_i \not \in \mathbb{Z}$ for some $i \in [m]$. In that case, the integer program is infeasible and we obtain the same conclusion by applying the GMIC cut to the $i^{\textup{th}}$ constraint. This is the GMIC corresponding to the aggregation multiplier $\lambda = e^i$. Thus, this GMIC is from the list $\Delta$ of aggregation multipliers, which gives the integer hull (and thus the GMIC closure).
  \item If $b_i + \gamma_i \in \mathbb{Z}$ for all $i \in [m]$, then the result follows from Theorem~\ref{thm:1}. 
  \end{itemize}
\end{proof}

Using techniques similar to the proof of Theorem~\ref{thm:1}, we can show a similar result about split closures where we define split closure (in an equivalent way) using only the left-hand-sides of the disjunctions~\cite{cook1990chvatal}:

\begin{theorem}\label{thm:split}
Let $\Gamma$  be the lattice generated by rational vectors $b^1, \dots, b^k \in \mathbb{Q}^m$. Let  $P(\gamma):= \{ x\in \mathbb{R}^n \,|\, Ax \leq b + \gamma\}$ for $\gamma \in \Gamma$ be a family of rational polyhedron and we are interested in split closures for $Q(\gamma) = P(\gamma) \cap \{x \,|\, x_j \in \mathbb{Z}, j \in J\}$. Let $P(\gamma)^{\pi} =\textup{conv}(P(\gamma) \cap \{x \in \mathbb{R}^n\,|\,\pi^{\top}x \in \mathbb{Z}\})$, where $\pi \in \mathbb{Z}^n$ and the support of $\pi$ belongs to $J$. Then there exists a finite set $\Pi:= \{\pi^1, \dots, \pi^g\}$, with $\pi^i \in \mathbb{Z}^n, \pi^i_j = 0 \ \forall j \not\in J$ for all $i \in [g]$, such that:
$$\textup{split closure}(P(\gamma)):= \bigcap_{\pi \in \mathbb{Z}^n, \pi_j = 0, j \not\in J}P(\gamma)^{\pi} = \bigcap_{\pi \in \Pi}P(\gamma)^{\pi}, \ \forall \gamma \in \Gamma.$$
\end{theorem}

\section{Cut Collection Algorithm} \label{section:algorthm}

For approximating the GMIC closure, we modified the relax-and-cut approach developed by  Fischetti and Salvagnin~\cite{fischetti2011relax}. The relax-and-cut approach is a method to generate a large number of simplex tableau corresponding to different basic feasible solutions which yield useful GMICs. 

In our implementation, we considerably  simplified the algorithm from that  originally proposed in~\cite{fischetti2011relax}.
Let $\text{C($B$)}$ denote the set of GMICs collected and retained from a tableau associated with a particular basis, $B$, of LP-relaxation of (\ref{eq:standard}). That is
\begin{equation*}
\text{C($B$)}\coloneq \left\{(\alpha^j_B)^{\top} x \geq \alpha_{0B}^j, \;\;\; j=1,\cdots, q(B)\right\}.
\end{equation*}

Our method iteratively goes through the following steps: 
\begin{enumerate}
\item 
Assuming we have already visited basis $B^1, \dots, B^k$, we have thus collected and retained the cuts in the set $\bigcup_{l = 1}^k \text{C($B^l$)}$. Let $c^{\top}x$ be the objective function of our IP. We solve the LP where we add all these GMICs, that is we solve:
\begin{eqnarray}\label{eq:LPCuts}
\begin{array}{rcl}
 &   \textup{min}& c^{\top}x \\
 &&Ax = b,  x\geq 0,\\
 && (\alpha^j_{B^l})^{\top} x \geq \alpha_{0{B^l}}^j \ \forall j = 1, \dots, q(B^l), \ \forall l \in [k].
 \end{array}
\end{eqnarray}
\item Let $u^j_{B^l}$ be an optimal dual solution of the above LP corresponding to the constraint $(\alpha^j_{B^l})^{\top} x \geq \alpha_{0{B^l}}^j$. 
We now construct the following Lagrangian relaxation to the previous LP:
\begin{eqnarray}\label{eq:Lag}
\begin{array}{rcl}
 &   \textup{min}& c^{\top}x + \sum_{l \in [k]}\sum_{j \in q(B^l)}u^j_{B^l} (\alpha_{0{B^l}}^j - (\alpha^j_{B^l})^{\top} x)\\
 &&Ax = b,  x\geq 0.
 \end{array}
\end{eqnarray}
Notice that the optimal objective function value of both the above LPs is equal. 

\item We solve (\ref{eq:Lag}), and let $B^{k +1}$ be the optimal basis. We collect a subset of GMICs corresponding to this optimal tableaux, call these cuts 
$C(B^{k +1})$. See details of which cuts are collected from a tableaux below. For cuts selected from the previously visited basis, we discard  those with zero dual value, i.e. we update $C(B)^l$ for $l \in \{1, \dots, k\}$ by removing cut $j$ if $u^j_{B^l} =0$ and update $q(B^l)$ accordingly. Then we  go to step 1, with $k \leftarrow k + 1$.
\end{enumerate}
The method is initialized by the GMICs from the optimal tableaux of the LP relaxation.
We terminate if the optimal objective function value of (\ref{eq:LPCuts}) for two consecutive iterations  does not change. 
Otherwise, we repeat the above iteration for at most $K$ iterations. In our experiments, we set $K = 10$. 

Finally, after $K$ rounds, we solve (\ref{eq:LPCuts}) and select the cuts with positive dual multipliers.  These GMICs are our approximation of the GMIC closure.

\paragraph{Rule for selection of GMICs from a tableaux.} We sort the right-hand-sides of the tableaux  with respect to the amount of fractionality (defined for $u \in \mathbb{R}$ as $\textup{min} \{u - \lfloor u \rfloor, \lceil u \rceil  - u\}$). Let $\hat{q}$ be the the number of rows with basic integer variables and the fractionality being greater than $0.001$. We let $q(B) = \textup{min} \{ \hat{q}, 500\}$. We then select the GMICs corresponding to the top $q(B)$ rows in terms of fractionality.

\begin{remark}
    Our method is considerably simpler to implement than the one proposed in Fischetti-Salvagnin~\cite{fischetti2011relax}. In particular, the paper~\cite{fischetti2011relax} presents various variants where potentially more basic feasible solutions are visited by solving the Lagrangian dual of (\ref{eq:LPCuts}) and generating cuts from all/subset of the optimal tablueax corresponding to solving (\ref{eq:Lag}) with $u$'s not being the optimal dual multiplier - but being the intermediary values obtained while solving the Lagrangian Dual of (\ref{eq:LPCuts}). However, as the paper states, this method requires significant engineering of the gradient updates to solve the Lagrangian dual. Moreover, as we see in the next section, we found that the cuts discovered by our simple implementation already improved the performance of a sizable number of instances. 
\end{remark}

\section{Benchmark Set Generation}\label{sec:InsPerturb}
\subsubsection{Right-hand-side changes.}

We perturbed the right-hand-sides in two steps. In the first step we used simple rules to produce a  preliminary perturbation of the right-hand-sides. In the second step, we checked if the first step produced any changes or not, or whether the perturbation led to infeasibilities.  

\paragraph{Step 1: Preliminary right-hand-side perturbation.}
We apply the following simple rules for perturbing the right-hand-side values:
\begin{itemize}
    \item \textbf{Rule 1:} If a constraint contains only two variables and exactly one of them is discrete, then perturb the right-hand-side value $b_i$ to $b_i\cdot r$ where $r$ is uniformly randomly selected from $[0.9,1.1]$.  
    \item \textbf{Rule 2:} If a constraint contains more than two variables and at least one of them is continuous, then perturb the right-hand-side value $b_i$ to $b_i\cdot r$ where $r$ is uniformly randomly selected  from $[0.9,1.1]$. We distinguish this rule from the previous, since we wanted to track how often the changes are made to constraints with two variables.  
    \item \textbf{Rule 3:} If a constraint contains only discrete variables, is not an equality constraint and the right-hand-side value is not equal to 1, then perturb the right-hand-side value $b_i$ to $b_i\cdot r$ where $r$ is uniformly randomly selected  from $[0.9,1.1]$.
    \item \textbf{Rule 4:} If a constraint contains only discrete variables and is an equality constraint, then we apply additive perturbation
    of $\{-1, 0, 1\}$ to the right-hand-side.
\end{itemize}
The above rules were designed based on the following guiding principles: 
\begin{enumerate}
   \item Prevent integer infeasibility to occur locally at the constraint level as far as possible:  \textbf{Rule 1} and \textbf{Rule 2} perturb the right-hand-side multiplicatively since a continuous variable is present. \textbf{Rule 3} perturbs the right-hand-side multiplicatively since the constraint is in inequality form. These perturbations may overlook  global implicit constraints on the RHS values. For example, the sum of demands should be lesser than sum of supplies when the constraints represent a network flow. 
    \item Preserve typical combinatorial constraints: special structures such as set partitioning, set covering, and set packing constraints are not changed; see \textbf{Rule 3} and \textbf{Rule 4}.
\end{enumerate}

\paragraph{Step 2: Checking the preliminary right-hand-side changes.}
The perturbations in the previous step may fail due to two reasons:
\begin{enumerate}
    \item Firstly, none of the constraints satisfy the requirements of \textbf{Rules 1}-\textbf{Rule 4}. In this case, we cannot make any right-hand-side changes to this MIPLIB instance. 
    \item Otherwise, we generate $5$ random perturbations. We, then check if these perturbations are integer feasible using Gurobi. If any perturbation is infeasible, then we do not make any right-hand-side changes to this MIPLIB instance. 
\end{enumerate}

\subsubsection{Objective changes.}
We perturb each objective coefficient $c_j$ to $c_j\cdot r$ where $r$ is uniformly randomly selected  from $[0.75,1.25]$. Note that, in some MIPLIB instances, the support of the objective function is on only one variable. Clearly, in this case, the above perturbation does not produce any change. Therefore, for such MIPLIB instances, we do not make any changes to the objective function.

\section{Supplementary Tables}
\include{1_missing}

\begin{landscape}
\include{2_perturbation}
\end{landscape}

\begin{table}[htbp]
  \centering
  \caption{Number of instances having $t$ distinct optimal solutions (integer part) across five perturbations.}
  \include{3_diversity_count}
  \label{tab:divcount}%
\end{table}%

\include{3_diversity}
\include{3_difficulty_mean}
\include{4_collect_avg}
\include{4_collect_fail}

\begin{landscape}
\include{5_expert_summary}
\include{6_ml_Work_Work}
\include{6_ml_Work_Speedup}
\include{6_ml_Time_Time}
\include{6_ml_Time_Speedup}
\end{landscape}

\end{document}

%% file: 6_ml_overview_all1.tex
\setlength{\tabcolsep}{8pt}
\begin{tabular}{lrrrrrr}
\toprule
 & Time (s) & Work & Nodes & Cuts & \multicolumn{2}{c}{Speedup} \\
\cmidrule{6-7} Method & & & &  & Time & Work \\
\midrule
ml:near:1 & 1,110.08 & 2,024.94 & 571,120.15 & 122.52 & 1.03 & 1.11 \\
expert & 1,154.88 & 1,924.89 & 701,039.03 & 122.17 & 1.02 & 1.20 \\
baseline & 1,219.13 & 2,035.39 & 719,464.26 & --- & 1.00 & 1.00 \\
ml:near:10 & 1,124.44 & 1,950.01 & 564,413.93 & 1,233.31 & 0.93 & 1.28 \\
ml:far:10 & 1,059.20 & 1,849.97 & 560,983.46 & 1,237.80 & 0.91 & 1.32 \\
ml:rand:10 & 1,111.68 & 1,918.12 & 583,215.51 & 1,235.16 & 0.91 & 1.30 \\
ml:near:50 & 1,160.96 & 1,807.49 & 490,234.73 & 6,147.18 & 0.65 & 1.49 \\
exp+col  & 1,556.93 & 1,924.89 & 701,039.03 & 122.17 & 0.46 & 1.20 \\
\bottomrule
\end{tabular}

%% file: 6_ml_overview_hard1.tex
\setlength{\tabcolsep}{7pt}
\begin{tabular}{lrrrrrr}
\toprule
 & Time (s) & Work & Nodes & Cuts & \multicolumn{2}{c}{Speedup} \\
\cmidrule{6-7} Method &  & & &  & Time & Work \\
\midrule
ml:far:10 & 3,581.44 & 6,378.97 & 1,904,188.59 & 1,190.49 & 1.17 & 1.22 \\
ml:rand:10 & 3,686.43 & 6,517.60 & 1,982,708.30 & 1,191.15 & 1.15 & 1.25 \\
ml:near:10 & 3,826.53 & 6,742.74 & 1,915,780.40 & 1,192.53 & 1.11 & 1.12 \\
expert & 3,981.47 & 6,650.68 & 2,406,732.23 & 116.19 & 1.11 & 1.12 \\
ml:near:50 & 3,648.58 & 6,240.77 & 1,652,977.92 & 5,847.93 & 1.08 & 1.40 \\
ml:near:1 & 3,791.68 & 6,950.06 & 1,938,210.00 & 118.19 & 1.08 & 1.02 \\
baseline & 4,200.30 & 7,012.50 & 2,464,156.55 & --- & 1.00 & 1.00 \\
exp+col & 4,568.36 & 6,650.68 & 2,406,732.23 & 116.19 & 0.91 & 1.12 \\
\bottomrule
\end{tabular}

%% file: 1_missing.tex
\begin{longtable}{ll}
\caption{List of eliminated instances and reason for elimination.} \label{tab:missing} \\
\toprule
instance & reason \\
\midrule
\endfirsthead
\caption[]{List of eliminated instances and reason for elimination.} \\
\toprule
instance & reason \\
\midrule
\endhead
\midrule
\multicolumn{2}{r}{Continued on next page} \\
\midrule
\endfoot
\bottomrule
\endlastfoot
blp-ar98 & no changes \\
bnatt400 & no changes \\
bnatt500 & infeasible \\
brazil3 & no changes \\
cryptanalysiskb128n5obj14 & infeasible \\
cryptanalysiskb128n5obj16 & no changes \\
csched008 & no changes \\
dano3\_3 & no changes \\
dano3\_5 & no changes \\
fhnw-binpack4-4 & infeasible \\
fhnw-binpack4-48 & no changes \\
fiball & no changes \\
germanrr & no changes \\
gfd-schedulen180f7d50m30k18 & no changes \\
highschool1-aigio & no changes \\
hypothyroid-k1 & no changes \\
leo1 & no changes \\
mcsched & no changes \\
neos-2075418-temuka & infeasible \\
neos-3004026-krka & no changes \\
neos-3402454-bohle & infeasible \\
neos-3988577-wolgan & infeasible \\
neos-5052403-cygnet & too large \\
neos-5104907-jarama & too large \\
neos-5114902-kasavu & too large \\
neos-848589 & too large \\
neos859080 & infeasible \\
ns1116954 & no changes \\
ns1952667 & no changes \\
physiciansched3-3 & no changes \\
physiciansched6-2 & no changes \\
piperout-08 & no changes \\
piperout-27 & no changes \\
proteindesign121hz512p9 & no changes \\
proteindesign122trx11p8 & no changes \\
radiationm18-12-05 & no changes \\
radiationm40-10-02 & no changes \\
rococoB10-011000 & no changes \\
rococoC10-001000 & no changes \\
roi5alpha10n8 & too large \\
roll3000 & no changes \\
s100 & too large \\
s250r10 & too large \\
savsched1 & too large \\
square41 & too large \\
square47 & too large \\
supportcase12 & too large \\
supportcase19 & too large \\
supportcase22 & too hard \\
swath1 & no changes \\
swath3 & no changes \\
traininstance2 & no changes \\
triptim1 & no changes \\
wachplan & no changes \\
\end{longtable}

%% file: 2_perturbation.tex
\setlength{\tabcolsep}{4pt}
\begin{longtable}{lrrrrrrrrr}
\caption{Instance size statistics and average number of changes of each type applied to generate the 55 instance variations.} \label{tab:perturbation} \\
\toprule
Instance & Variables & Constraints & Obj. Changes & \multicolumn{5}{c}{RHS Changes} & Changes \\
\cmidrule{5-9} &  &  &  & Type 1 & Type 2 & Type 3 & Type 4 & Total &  \\
\midrule
\endfirsthead
\caption[]{Instance size statistics and average number of changes of each type applied to generate 55 perturbations.} \\
\toprule
Instance & Variables & Constraints & Obj. changes & \multicolumn{5}{c}{RHS changes} & Changes \\
\cmidrule{5-9} &  &  &  & Type 1 & Type 2 & Type 3 & Type 4 & Total &  \\
\midrule
\endhead
\midrule
\multicolumn{10}{r}{Continued on next page} \\
\midrule
\endfoot
\bottomrule
\endlastfoot
30n20b8 & 18,380 & 576 & 2 & 0 & 0 & 0 & 0 & 0 & 2 \\
50v-10 & 2,013 & 233 & 1,647 & 0 & 0 & 0 & 0 & 0 & 1,647 \\
CMS750\_4 & 11,697 & 16,381 & 751 & 0 & 6,446 & 1,688 & 0 & 8,134 & 8,885 \\
academictimetablesmall & 28,926 & 23,294 & 255 & 0 & 0 & 0 & 0 & 0 & 255 \\
air05 & 7,195 & 426 & 7,195 & 0 & 0 & 0 & 0 & 0 & 7,195 \\
app1-1 & 2,480 & 4,926 & 23 & 0 & 0 & 0 & 0 & 0 & 23 \\
app1-2 & 26,871 & 53,467 & 266 & 0 & 0 & 0 & 0 & 0 & 266 \\
assign1-5-8 & 156 & 161 & 26 & 0 & 0 & 0 & 0 & 0 & 26 \\
\end{longtable}

%% file: 3_diversity_count.tex
\begin{tabular}{rr}
\toprule
Solutions (t) & Num. of instances \\
\midrule
1 & 8 \\
2 & 6 \\
3 & 7 \\
4 & 8 \\
5 & 157 \\
\bottomrule
\end{tabular}

%% file: 3_diversity.tex
\setlength{\tabcolsep}{4pt}
\begin{longtable}{lrrrr}
\caption{Number of obj. and RHS changes applied to each instance, and number of distinct optimal solutions and values obtained across five perturbations.} \label{tab:diversity} \\
\toprule
Instance & Obj. changes & RHS changes & Solutions & Opt. values \\
\midrule
\endfirsthead
\caption[]{Number of obj. and RHS changes applied to each instance, and number of distinct optimal solutions and values obtained across five perturbations.} \\
\toprule
Instance & Obj. changes & RHS changes & Solutions & Opt. values \\
\midrule
\endhead
\midrule
\multicolumn{5}{r}{Continued on next page} \\
\midrule
\endfoot
\bottomrule
\endlastfoot
30n20b8 & 2 & 0 & 5 & 5 \\
50v-10 & 1,647 & 0 & 5 & 5 \\
CMS750\_4 & 751 & 8,134 & 5 & 5 \\
academictimetablesmall & 255 & 0 & 5 & 1 \\
air05 & 7,195 & 0 & 5 & 5 \\
app1-1 & 23 & 0 & 5 & 5 \\
app1-2 & 266 & 0 & 1 & 5 \\
assign1-5-8 & 26 & 0 & 5 & 5 \\
\end{longtable}

%% file: 3_difficulty_mean.tex
\begin{longtable}{lrrrr}
\caption{Average running time and nodes required to solve original and perturbed MIPLIB 2017 instances.} \label{tab:difficultymean} \\
\toprule
Instance & \multicolumn{2}{c}{MIP Time (s)} & \multicolumn{2}{c}{B\&B Nodes} \\
 & Original & Perturbed & Original & Perturbed \\
\midrule
\endfirsthead
\caption[]{Average running time and nodes required to solve original and perturbed MIPLIB 2017 instances.} \\
\toprule
Instance & \multicolumn{2}{c}{MIP Time (s)} & \multicolumn{2}{c}{B\&B Nodes} \\
 & Original & Perturbed & Original & Perturbed \\
\midrule
\endhead
\midrule
\multicolumn{5}{r}{Continued on next page} \\
\midrule
\endfoot
\bottomrule
\endlastfoot
30n20b8 & 1.1 & 7.8 & 271.0 & 168.2 \\
50v-10 & 0.0 & 8,802.6 & 0.0 & 6,998,394.2 \\
CMS750\_4 & 1.9 & 30.6 & 368.0 & 4,114.8 \\
academictimetablesmall & 314.7 & 665.5 & 1,956.0 & 5,159.8 \\
air05 & 3.0 & 6.2 & 569.0 & 1,095.8 \\
app1-1 & 0.3 & 0.4 & 1.0 & 90.8 \\
app1-2 & 1.7 & 59.5 & 1.0 & 1,219.6 \\
assign1-5-8 & 0.0 & 5,026.3 & 0.0 & 11,980,898.8 \\
\end{longtable}

%% file: 4_collect_avg.tex
\begin{longtable}{lrrrrr}
\caption{Cut collection statistics for the five test variations.} \label{tab:collect} \\
\toprule
Instance & Time (s) & \multicolumn{3}{c}{MIP Gap (\%)} & Success (\%) \\
\cmidrule{3-5} &  & Initial & Final & Delta &  \\
\midrule
\endfirsthead
\caption[]{Cut collection statistics for the five test variations.} \\
\toprule
Instance & Time (s) & \multicolumn{3}{c}{MIP Gap (\%)} & Success (\%) \\
\cmidrule{3-5} &  & Initial & Final & Delta &  \\
\midrule
\endhead
\midrule
\multicolumn{6}{r}{Continued on next page} \\
\midrule
\endfoot
\bottomrule
\endlastfoot
30n20b8 & 16.27 & 99.47 & 50.00 & 49.46 & 100.00 \\
50v-10 & 26.82 & 15.18 & 6.93 & 8.25 & 100.00 \\
CMS750\_4 & --- & --- & --- & --- & 0.00 \\
academictimetablesmall & --- & --- & --- & --- & 0.00 \\
air05 & 56.55 & 2.41 & 2.17 & 0.24 & 100.00 \\
app1-1 & 34.78 & 700.19 & 410.10 & 290.08 & 100.00 \\
app1-2 & 187.40 & 546.55 & 543.87 & 2.68 & 100.00 \\
assign1-5-8 & 35.86 & 12.95 & 11.85 & 1.11 & 100.00 \\
\end{longtable}

%% file: 4_collect_fail.tex
\begin{longtable}{lccccc}
\caption{Number of test variations for which the cut collection procedure failed, grouped by failure reason.} \label{tab:collectfail} \\
\toprule
Instance & Memory & Numerics & Time Limit & Zero Cuts & Total \\
\midrule
\endfirsthead
\caption[]{Number of test variations for which the cut collection procedure failed, grouped by failure reason.} \\
\toprule
Instance & Memory & Numerics & Time Limit & Zero Cuts & Total \\
\midrule
\endhead
\midrule
\multicolumn{6}{r}{Continued on next page} \\
\midrule
\endfoot
\bottomrule
\endlastfoot
CMS750\_4 & --- & --- & --- & 5 & 5 \\
academictimetablesmall & --- & --- & --- & 5 & 5 \\
bppc4-08 & --- & --- & --- & 5 & 5 \\
cbs-cta & --- & --- & --- & 5 & 5 \\
chromaticindex1024-7 & --- & --- & --- & 5 & 5 \\
chromaticindex512-7 & --- & --- & --- & 5 & 5 \\
comp07-2idx & --- & --- & --- & 1 & 1 \\
decomp2 & --- & --- & --- & 5 & 5 \\
\end{longtable}

%% file: 5_expert_summary.tex
\begin{longtable}{lrrrrrrrl}
\caption{Average performance of expert method over five test variations. Instances with failed variations have been discarded.} \label{tab:expert} \\
\toprule
 & \multicolumn{2}{c}{Work} & \multicolumn{2}{c}{Nodes} & \multicolumn{2}{c}{Gap (\%)} & Speedup & Outcome \\
Method & baseline & expert & baseline & expert & baseline & expert & expert &  \\
Instance &  &  &  &  &  &  &  &  \\
\midrule
\endfirsthead
\caption[]{Average performance of expert method over five test variations. Instances with failed variations have been discarded.} \\
\toprule
 & \multicolumn{2}{c}{Work} & \multicolumn{2}{c}{Nodes} & \multicolumn{2}{c}{Gap (\%)} & Speedup & Outcome \\
Method & baseline & expert & baseline & expert & baseline & expert & expert &  \\
Instance &  &  &  &  &  &  &  &  \\
\midrule
\endhead
\midrule
\multicolumn{9}{r}{Continued on next page} \\
\midrule
\endfoot
\bottomrule
\endlastfoot
30n20b8 & 18.24 & 5.29 & 253.33 & 27.67 & 0.00 & 0.00 & 3.448 & Positive \\
neos-4532248-waihi & 288.41 & 120.78 & 2,546.53 & 1,268.60 & 0.00 & 0.00 & 2.388 & Positive \\
supportcase6 & 185.38 & 107.46 & 1,185.60 & 843.60 & 0.00 & 0.00 & 1.725 & Positive \\
exp-1-500-5-5 & 4.02 & 2.42 & 3,992.53 & 2,142.53 & 0.00 & 0.00 & 1.664 & Positive \\
neos-5049753-cuanza & 1,298.35 & 820.43 & 1,087.87 & 978.67 & 0.00 & 0.00 & 1.583 & Positive \\
mik-250-20-75-4 & 0.13 & 0.09 & 371.60 & 108.40 & 0.00 & 0.00 & 1.362 & Positive \\
irp & 0.46 & 0.35 & 82.27 & 78.27 & 0.00 & 0.00 & 1.340 & Positive \\
supportcase40 & 602.30 & 463.11 & 37,040.20 & 37,672.33 & 0.00 & 0.00 & 1.301 & Positive \\
\end{longtable}

%% file: 6_ml_Work_Work.tex
\begin{longtable}{lrrrrrrr}
\caption{Average work units of different methods on the 50 benchmark instances.} \label{tab:mlWork1} \\
\toprule
 & \multicolumn{7}{c}{Work} \\
Method & baseline & expert & ml:far:10 & ml:near:1 & ml:near:10 & ml:near:50 & ml:rand:10 \\
Instance &  &  &  &  &  &  &  \\
\midrule
\endfirsthead
\caption[]{Average work units of different methods on the 50 benchmark instances.} \\
\toprule
 & \multicolumn{7}{c}{Work} \\
Method & baseline & expert & ml:far:10 & ml:near:1 & ml:near:10 & ml:near:50 & ml:rand:10 \\
Instance &  &  &  &  &  &  &  \\
\midrule
\endhead
\midrule
\multicolumn{8}{r}{Continued on next page} \\
\midrule
\endfoot
\bottomrule
\endlastfoot
30n20b8 & 18.24 & 5.29 & 4.84 & 5.29 & 4.84 & 4.84 & 4.67 \\
neos-4532248-waihi & 288.41 & 120.78 & 120.78 & 120.78 & 120.78 & 120.78 & 120.78 \\
supportcase6 & 185.38 & 107.46 & 109.41 & 118.19 & 112.15 & 105.69 & 118.24 \\
exp-1-500-5-5 & 4.02 & 2.42 & 0.78 & 3.51 & 1.02 & 0.39 & 0.88 \\
neos-5049753-cuanza & 1,298.35 & 820.43 & 558.99 & 1,211.00 & 584.02 & 362.67 & 529.75 \\
mik-250-20-75-4 & 0.13 & 0.09 & 0.08 & 0.08 & 0.09 & 0.09 & 0.08 \\
irp & 0.46 & 0.35 & 0.43 & 0.46 & 0.44 & 0.40 & 0.44 \\
supportcase40 & 602.30 & 463.11 & 528.64 & 467.38 & 543.01 & 563.78 & 525.53 \\
\end{longtable}

%% file: 6_ml_Work_Speedup.tex
\begin{longtable}{lrrrrrr}
\caption{Average speedup (with respect to work units) of different methods on the 50 benchmark instances.} \label{tab:mlWork2} \\
\toprule
 & \multicolumn{6}{c}{Speedup} \\
Method & expert & ml:far:10 & ml:near:1 & ml:near:10 & ml:near:50 & ml:rand:10 \\
Instance &  &  &  &  &  &  \\
\midrule
\endfirsthead
\caption[]{Average speedup (with respect to work units) of different methods on the 50 benchmark instances.} \\
\toprule
 & \multicolumn{6}{c}{Speedup} \\
Method & expert & ml:far:10 & ml:near:1 & ml:near:10 & ml:near:50 & ml:rand:10 \\
Instance &  &  &  &  &  &  \\
\midrule
\endhead
\midrule
\multicolumn{7}{r}{Continued on next page} \\
\midrule
\endfoot
\bottomrule
\endlastfoot
30n20b8 & 3.45 & 3.77 & 3.45 & 3.77 & 3.77 & 3.90 \\
neos-4532248-waihi & 2.39 & 2.39 & 2.39 & 2.39 & 2.39 & 2.39 \\
supportcase6 & 1.73 & 1.69 & 1.57 & 1.65 & 1.75 & 1.57 \\
exp-1-500-5-5 & 1.66 & 5.15 & 1.14 & 3.96 & 10.44 & 4.55 \\
neos-5049753-cuanza & 1.58 & 2.32 & 1.07 & 2.22 & 3.58 & 2.45 \\
mik-250-20-75-4 & 1.36 & 1.60 & 1.56 & 1.49 & 1.47 & 1.51 \\
irp & 1.34 & 1.08 & 1.00 & 1.05 & 1.17 & 1.05 \\
supportcase40 & 1.30 & 1.14 & 1.29 & 1.11 & 1.07 & 1.15 \\
\end{longtable}

%% file: 6_ml_Time_Time.tex
\begin{longtable}{lrrrrrrr}
\caption{Average time of different methods on the 14 hard benchmark instances.} \label{tab:mlTime1} \\
\toprule
 & \multicolumn{7}{c}{Time (s)} \\
Method & baseline & exp+col & ml:far:10 & ml:near:1 & ml:near:10 & ml:near:50 & ml:rand:10 \\
Instance &  &  &  &  &  &  &  \\
\midrule
\endfirsthead
\caption[]{Average time of different methods on the 14 hard benchmark instances.} \\
\toprule
 & \multicolumn{7}{c}{Time (s)} \\
Method & baseline & exp+col & ml:far:10 & ml:near:1 & ml:near:10 & ml:near:50 & ml:rand:10 \\
Instance &  &  &  &  &  &  &  \\
\midrule
\endhead
\midrule
\multicolumn{8}{r}{Continued on next page} \\
\midrule
\endfoot
\bottomrule
\endlastfoot
50v-10 & 6,504.97 & 6,576.95 & 6,258.86 & 6,273.33 & 6,080.06 & 6,363.89 & 6,198.39 \\
bab2 & 2,982.14 & 6,945.93 & 2,442.47 & 3,894.55 & 4,546.77 & 2,340.32 & 1,860.93 \\
cmflsp50-24-8-8 & 5,561.02 & 6,239.72 & 5,121.11 & 5,118.99 & 5,127.74 & 5,382.54 & 5,099.62 \\
comp21-2idx & 3,304.79 & 3,294.17 & 3,085.97 & 3,274.97 & 3,391.48 & 3,297.29 & 3,143.48 \\
dws008-01 & 3,125.89 & 3,014.31 & 2,892.82 & 2,692.95 & 2,916.08 & 2,907.35 & 2,985.44 \\
neos-3046615-murg & 4,390.31 & 4,169.14 & 3,888.79 & 4,437.08 & 4,062.52 & 3,853.29 & 3,887.18 \\
neos-3627168-kasai & 12,509.69 & 11,740.51 & 8,366.29 & 7,792.25 & 9,355.04 & 6,514.19 & 9,445.87 \\
neos-4954672-berkel & 11,622.77 & 11,254.44 & 9,899.46 & 10,485.50 & 10,313.65 & 10,594.96 & 10,393.03 \\
\end{longtable}

%% file: 6_ml_Time_Speedup.tex
\begin{longtable}{lrrrrrr}
\caption{Average speedup (with respect to time) of different methods on the 14 hard benchmark instances.} \label{tab:mlTime2} \\
\toprule
 & \multicolumn{6}{c}{Speedup} \\
Method & exp+col & ml:far:10 & ml:near:1 & ml:near:10 & ml:near:50 & ml:rand:10 \\
Instance &  &  &  &  &  &  \\
\midrule
\endfirsthead
\caption[]{Average speedup (with respect to time) of different methods on the 14 hard benchmark instances.} \\
\toprule
 & \multicolumn{6}{c}{Speedup} \\
Method & exp+col & ml:far:10 & ml:near:1 & ml:near:10 & ml:near:50 & ml:rand:10 \\
Instance &  &  &  &  &  &  \\
\midrule
\endhead
\midrule
\multicolumn{7}{r}{Continued on next page} \\
\midrule
\endfoot
\bottomrule
\endlastfoot
50v-10 & 0.99 & 1.04 & 1.04 & 1.07 & 1.02 & 1.05 \\
bab2 & 0.43 & 1.22 & 0.77 & 0.66 & 1.27 & 1.60 \\
cmflsp50-24-8-8 & 0.89 & 1.09 & 1.09 & 1.08 & 1.03 & 1.09 \\
comp21-2idx & 1.00 & 1.07 & 1.01 & 0.97 & 1.00 & 1.05 \\
dws008-01 & 1.04 & 1.08 & 1.16 & 1.07 & 1.08 & 1.05 \\
neos-3046615-murg & 1.05 & 1.13 & 0.99 & 1.08 & 1.14 & 1.13 \\
neos-3627168-kasai & 1.07 & 1.50 & 1.61 & 1.34 & 1.92 & 1.32 \\
neos-4954672-berkel & 1.03 & 1.17 & 1.11 & 1.13 & 1.10 & 1.12 \\
neos-5049753-cuanza & 0.47 & 1.38 & 0.93 & 1.23 & 0.42 & 1.45 \\
nursesched-medium-hint03 & 0.83 & 0.88 & 0.69 & 1.16 & 1.01 & 0.83 \\
sp97ar & 0.94 & 1.39 & 1.16 & 1.23 & 1.16 & 1.13 \\
supportcase40 & 1.06 & 1.18 & 1.36 & 1.25 & 0.84 & 1.17 \\
trento1 & 0.89 & 1.03 & 1.06 & 1.06 & 1.02 & 1.03 \\
uct-subprob & 1.01 & 1.20 & 1.16 & 1.18 & 1.17 & 1.02 \\
Arithmetic Average & 0.91 & 1.17 & 1.08 & 1.11 & 1.08 & 1.15 \\
\end{longtable}